\documentclass{amsart}
\usepackage{amssymb}

\newcommand{\U}{\mathcal U}
\newcommand{\V}{\mathcal V}

\newcommand{\Hol}{\mathrm{Hol}}
\newcommand{\IZ}{\mathbb Z}
\newcommand{\Ra}{\Rightarrow}

\newtheorem{theorem}{Theorem}
\newtheorem{proposition}{Proposition}

\newtheorem{problem}{Problem}
\newtheorem{example}{Example}
\newtheorem{question}{Question}
\newtheorem{lemma}{Lemma}
\theoremstyle{definition}
\newtheorem{remark}{Remark}
\begin{document}
\title[Subsemigroups of the hyperspaces over topological groups]{Embedding topological semigroups into the hyperspaces over topological groups}
\author{Taras Banakh, Olena Hryniv}
\begin{abstract} We study algebraic and topological properties of subsemigroups of the hyperspace $\exp(G)$ of non-empty compact subsets of a topological group $G$ endowed with the Vietoris topology and the natural semigroup operation. On this base we prove that a compact Clifford topological semigroup $S$ is topologically isomorphic to a subsemigroup of $\exp(G)$ for a suitable topological group $G$ if and only if $S$ is a topological inverse semigroup with zero-dimensional idempotent semilattice. 
\end{abstract}
\address{Uniwersytet Humanistyczno-Przyrodniczy Jana Kochanowskiego, Kielce, Poland \newline and  
 Department of Mathematics, Lviv National University, Lviv, Ukraine}
\email{tbanakh@yahoo.com, olena\_hryniv@ukr.net}
\maketitle

\section{Introduction}

According to \cite{Ber} (and \cite{Trn}) each (commutative) semigroup $S$ embeds into the global semigroup $\Gamma(G)$ over a suitable (Abelian) group $G$. The global semigroup $\Gamma(G)$ over $G$ is the set of all non-empty subsets of $G$ endowed with the semigroup operation $(A,B)\mapsto AB=\{ab:a\in A,\; b\in B\}$. If $G$ is a topological group, then the global semigroup $\Gamma(G)$ contains a subsemigroup $\exp(G)$ consisting of all non-empty compact subsets of $G$ and carrying a natural topology turning it into a topological semigroup. This is the Vietoris topology generated by the sub-base consisting of the sets $$U^+=\{K\in\exp(S):K\subset U\}\mbox{ and }U^-=\{K\in\exp(S):K\cap U\ne\emptyset\}$$ where $U$ runs over open subsets of $S$.
Endowed with the Vietoris topology the semigroup $\exp(G)$ will be referred to as the {\em hypersemigroup} over $G$ (because its underlying topological space is the hyperspace $\exp(G)$ of $G$, see \cite{TZ}). Since each topological group $G$ is Tychonov, so is the hypersemigroup $\exp(G)$. The group $G$ can be identified with the subgroup $\{K\in\exp(G):|K|=1\}$ of $\exp(G)$ consistsing of singletons.

The main object of our study in this paper is the class $\mathcal H$ of topological semigroups $S$ that embed into the hypersemigroups $\exp(G)$ over topological groups $G$. We shall say that a topological semigroup $S_1$ embeds into another topological semigroup $S_2$ if there is a semigroup homomorphism $h:S_1\to S_2$ that is a topological embedding.
In is clear that the class $\mathcal H$ contains all topological groups. On the other hand, the compact topological semigroup $([0,1],\min)$ does not belong to $\mathcal H$, see \cite{BL}. In this paper we establish some inheritance properties of the class $\mathcal H$ and on this base detect compact Clifford semigroups belonging to $\mathcal H$: those are precisely compact Clifford inverse semigroups with zero-dimensional idempotent semilattice.

Let us recall that a semigroup $S$ is 
\begin{itemize}
\item {\em Clifford} if each element $x\in S$ lies in a subgroup of $S$; 
\item {\em inverse} if each element $x\in S$ is {\em uniquely invertible} in the sense that there is a unique element $x^{-1}\in S$ called the {\em inverse} of $x$ such that 
$xx^{-1}x=x$ and $x^{-1}xx^{-1}=x^{-1}$;
\item {\em algebraically regular} if each element $x\in S$ is {\em regular} in the sense that $xyx=x$ for some $y\in S$;
\item a {\em semilattice} if $xx=x$ and $xy=yx$ for all $x,y\in S$. 
\end{itemize}
It is known \cite[1.17]{CP}, \cite[II.1.2]{Pet} that a semigroup $S$ is inverse if and only if $S$ is algebraically regular and the set $E=\{x\in S:xx=x\}$ of idempotents is a commutative subsemigroup of $S$. The subsemigroup $E$ will be called the {\em idempotent semilattice} of $S$. 
An inverse semigroup $S$ is Clifford if and only if $xx^{-1}=x^{-1}x$ for all $x\in S$.
In this case $S=\bigcup_{e\in E}H_e$ where $H_e=\{x\in S:xx^{-1}=e=x^{-1}x\}$ are the maximal subgroups of $S$ corresponding to the idempotents $e$ of $S$.

The above classes of semigroups relate as follows:

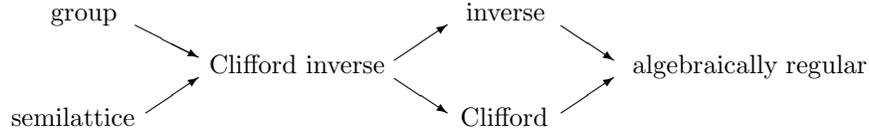
\begin{picture}(300,70)(-10,-10)
\put(10,40){group}
\put(42,38){\vector(2,-1){24}}
\put(-5,0){semilattice}
\put(46,5){\vector(3,2){20}}
\put(70,20){Clifford inverse}
\put(140,25){\vector(3,2){20}}
\put(140,18){\vector(3,-2){20}}
\put(167,40){inverse}
\put(203,38){\vector(3,-2){20}}
\put(165,0){Clifford}
\put(203,5){\vector(3,2){20}}
\put(230,20){algebraically regular}
\end{picture}

These classes form varieties of semigroups, which means that they are closed under taking subdirect products and homomorphic images. As we shall see later, the class $\mathcal H$ is not closed under homomorphic images and thus does not form a variety but is invariant with respect to many operations over topological semigroups.

By a {\em topological semigroup} we understand a topological space $S$ endowed with a continuous semigroup operation. A topological semigroup $S$ is called a {\em topological inverse semigroup} if $S$ is an inverse semigroup and the inversion map $(\cdot)^{-1}:S\to S$, $(\cdot)^{-1}:x\mapsto x^{-1}$ is continuous.  

Now we define three operations over topological semigroups that do not lead out the class $\mathcal H$. 

We say that a topological semigroup $S$ is a {\em subdirect product} of a family $\{S_\alpha:\alpha\in A\}$ of topological semigroups if $S$ embeds into the Tychonov product $\prod_{\alpha\in A}S_\alpha$ endowed the coordinatewise semigroup operation. 

Another operation is the {\em semidirect product} $S\leftthreetimes^\sigma G$ of a topological semigroup $S$ and a topological group $G$ acting on $S$ by authomorphisms. More precisely, let $Aut(S)$ denote the group of topological auto-isomorphisms of the semigroup $S$  and $\sigma:G\to Aut(S)$ be a group homomorphism defined on a topological group $G$ and such that the induced map $$\widetilde\sigma:G\times S\to S,\;\tilde\sigma:(g,s)\mapsto \sigma(g)(s)$$is continuous. By $S \leftthreetimes^\sigma G$ we denote the topological semigroup whose underlying topological space is the Tychonov product $S\times G$ and the semigroup operation is given by the formula $(s,g)*(s',g')=(sg(s'),gg')$. The semidirect product $S\leftthreetimes^{id}Aut(S)$ of a semigroup $S$ with its automorphism group is called {\em the holomorph} of $S$ and is denoted by $\Hol(S)$.

One can easily check that for an algebraically  regular (inverse) topological semigroup $S$ the semidirect product $S\leftthreetimes^\sigma G$ with any topological group $G$ acting on $S$ is an algebraically regular (inverse) topological semigroup. The situation is different for Clifford topological semigroups: the semidirect product $S\leftthreetimes^\sigma G$ is Clifford inverse if and only if $S$ is Clifford inverse and the group $G$ acts trivially on the idempotents of $S$, see Proposition~\ref{semidir}.

The third operation that does not lead out the class $\mathcal H$ is attaching zero to a compact semigroup from $\mathcal H$. Given a topological semigroup $S$ let $S^0=S\cup\{0\}$ denote the extension of $S$ by an isolated point $0\notin S$ such that $s0=0s=0$ for all $s\in S^0$. 

\begin{theorem}\label{t1} The class $\mathcal H$ is closed under the following three operations:
\begin{enumerate}
\item subdirect products;
\item semidirect products with Abelian topological groups;
\item attaching zero to compact semigroups from $\mathcal H$.
\end{enumerate}
\end{theorem}

\begin{proof} 1. The first item follows from the fact that for any family $\{H_\alpha\}_{\alpha\in A}$ of topological groups the map
$$E:\prod_{\alpha\in A}\exp(H_\alpha)\to \exp(\prod_{\alpha\in A}H_\alpha),\quad E:(K_\alpha)_{\alpha\in A}\mapsto \prod_{\alpha\in A}K_\alpha$$
is an embedding of topological semigroups.
\smallskip

2. The second item is less trivial and will be proved in Section~\ref{s3}.
\smallskip

3. If $S\in\mathcal H$ is a compact semigroup, then there is an embedding $f:S\to\exp(H)$ of $S$ into the hypersemigroup $\exp(H)$ of some compact topological group $H$, see Proposition~\ref{p1} below. Take any compact topological group $G$ containing $H$ so that $H\ne G$ and define the map $f^0:S^0\to \exp(G)$ letting $f^0|S=f$ and $f^0(0)=G$. It can be shown that $f^0$ is a topological embedding and thus $S^0\in\mathcal H$.
\end{proof}

\begin{problem} Is the class $\mathcal H$ closed under taking semidirect products with arbitrary (not necessarily abelian) topological groups?
\end{problem}

In light of Theorem~\ref{t1} it is natural to consider the smallest class $\mathcal H_0$  of topological semigroups, closed under subdirect products, semidirect products with Abelian topological groups and attaching zero to compact semigroups from $\mathcal H$. 
Since the class of topological inverse semigroups is closed under those three operations, we conclude that $\mathcal H_0$ is a subclass of the class of topological inverse semigroups. Consequently, $\mathcal H_0$ is strictly smaller than the class $\mathcal H$ (because for a topological group $G$ the semigroup $\exp(G)$ is inverse if and only if $|G|\le2$).

Nonetheless we can ask the following 

\begin{question}\label{q1} Does each (compact) topological inverse semigroup $S\in\mathcal H$ belong to the class $\mathcal H_0$?
\end{question}

In this respect let us note the following property of compact semigroups from the class $\mathcal H$.

\begin{proposition}\label{p1} A compact topological semigroup $S$ belongs to the class $\mathcal H$ if and only if $S$ embeds into the hypersemigroup $\exp(G)$ over a compact topological group $G$.
\end{proposition}

\begin{proof} Given a compact topological semigroup $S \in \mathcal H$ find  an embedding $h: S \to exp(G)$ of $S$ into the hypersemigroup $\exp(G)$ over a topological group $G$. It follows from \cite[2.1.2]{TZ} that the union $H=\bigcup_{s\in S}h(s)\subset G$ is compact. Moreover, $H$ is a subsemigroup of $G$. Indeed, given arbitrary points $y,y'\in H$ find points $x,x'\in S$ with $y\in h(x)$ and $y'\in h(x')$. Then $yy'\in h(x)h(x')=h(xx')\subset H$. Being a compact cancellative semigroup, $H$ is a topological group by \cite[Th.1.10]{CHK1}. Since $h(S)\subset\exp(H)\subset\exp(G)$, we see that $S$ embeds into the hypersemigroup $\exp(H)$ over the compact topological group $H$.
\end{proof}

We shall affirmatively answer the ``compact'' part of Question~\ref{q1} under an additional assumption that $S\in\mathcal H$ is Clifford.
For this we first establish some specific algebraic and topological properties of algebraically regular semigroups $S\in\mathcal H$.

Let us call two elements $x,y$ of an inverse semigroup $S$ {\em conjugated} if $x=zyz^{-1}$ and $y=z^{-1}xz$ for some element $z\in S$. 
 For an element $e\in E$ of a semilattice $E$ let ${\uparrow}e=\{f\in E:ef=e\}$ denote the principal filter of $e$. We say that two elements $e,f\in E$ are {\em incomparable} if their product $ef$ differs from $e$ and $f$ (this is equivalent to $e\notin{\uparrow}f$ and $f\notin{\uparrow}e$).

A topological space $X$ is called 
\begin{itemize}
\item {\em totally disconnected} if for any distinct points $x,y\in X$ there is a closed-and-open subset $U\subset X$ containing $x$ but not $y$;
\item {\em zero-dimensional} if the family of closed-and-open sets forms a base of the topology of $X$.
\end{itemize}
It is known that a compact Hausdorff space is zero-dimensional if and only if it is totally disconnected.

\begin{theorem}\label{t2} If a topological semigroup $S\in\mathcal H$ is algebraically regular, then 
\begin{enumerate}
\item $S$ is a topological inverse semigroup;
\item the idempotent semilattice $E$ of $S$ has totally disconnected principal filters ${\uparrow}e$, $e\in E$;
\item an element $x\in S$ is an idempotent if and only if $x^2x^{-1}$ is an idempotent;
\item any distinct conjugated idempotents of $S$ are incomparable.
\end{enumerate}
\end{theorem} 

This theorem will be proved in Section~\ref{s5}.

\begin{remark} Theorem~\ref{t2} allows us to construct many examples of  algebraically regular topological semigroups non-embeddable into the hypersemigroups over  topological groups. The first two items of this proposition imply the result of \cite{BL} that non-trivial rectangular semigroups and connected topological semilattices do not belong to the class $\mathcal H$. The last two items imply that the class $\mathcal H$ does not contain neither  Brandt  nor bicyclic semigroups. A {\em bicyclic semigroup} is a semigroup generated by two elements $p,q$ connected by the relation $qp=1$. 

By a {\em Brandt semigroup} we understand a semigroup of the form $$B(H,\kappa)=(\kappa\times H\times \kappa)\cup\{0\}$$ where $H$ is a group, $\kappa$ is a non-empty set, and the product $(\alpha,h,\beta)*(\alpha',h',\beta')$ of two non-zero elements of $B(H,\kappa)$ is equal to $(\alpha,hh',\beta')$ if $\beta=\alpha'$ and $0$ otherwise.  Brandt and bicyclic semigroups play an important role in the structure theory of inverse semigroups, see \cite{Pet}. 
\end{remark}

The following theorem answers affirmatively the ``compact'' part of Question~\ref{q1}.

\begin{theorem}\label{t5} For a compact topological Clifford semigroup $S$ the following conditions are equivalent:
\begin{enumerate}
\item $S$ belongs to the class $\mathcal H$;
\item $S$ belongs to the class $\mathcal H_0$;
\item $S$ is a topological inverse semigroup with zero-dimensional idempotent semilattice $E$;
\item $S$ embeds into the product $\prod_{e\in E}H_e^\circ$;
\item $S$ embeds into the hypersemigroup $\exp(G)$ of the compact topological group $G=\prod_{e\in E}\widetilde H_e$, where for each idempotent $e\in E$ \ $\widetilde H_e$ is a non-trivial compact topological group containing the maximal group $H_e$.
\end{enumerate}
\end{theorem}

\begin{proof} It suffices to prove the implications: $(4)\Ra(5)\Ra(2)\Ra(1)\Ra(3)\Ra(4)$ among which $(5)\Ra(2)\Ra(1)$ are trivial.
\smallskip

 $(4)\Ra(5)$. Assume that $S$ embeds into the product $\prod_{e\in E}H_e^0$. For each idempotent $e\in E$ fix a non-trivial compact topological group $\widetilde H_e$ containing $H_e$ and define an embedding $f_e:H^0_e\to\exp(\widetilde H_e)$ letting $f_e(h)=\{h\}$ if $h\in H_e$ and $f_e(0)=\widetilde H_e$. 

The product of embeddings $f_e$, $e\in E$, yields embeddings 
$$S\to\prod_{e\in E}H_e^0\to \prod_{e\in E}\exp(\widetilde H_e)\to \exp(\prod_{e\in E}\widetilde H_e)$$ the latter homomorphism defined by $$\prod_{e\in E}\exp(\widetilde H_e)\ni(K_e)_{e\in E}\mapsto \prod_{e\in E}K_e\in\exp(\prod_{e\in E}\widetilde H_e).$$
\smallskip

$(1)\Ra(3)$ Assume that $S\in\mathcal H$. Then $S$ is a compact topological inverse semigroup according to Theorem~\ref{t2}(1). The semigroup $E$ of idempotents of $S$ is compact and thus contains the smallest idempotent $e\in E$ (in the sense that $ee'=e$ for all $e'\in E$). By Theorem~\ref{t2}, the principal filter ${\uparrow}e=E$ is totally disconnected and being compact, is zero-dimensional.
\smallskip

$(3)\Ra(4)$ Assume that $S$ is a compact topological inverse Clifford semigroup with zero-dimensional idempotent semilattice $E$. Let $\pi:S\to E$, $\pi:x\mapsto xx^{-1}=x^{-1}x$ be  the retraction of $S$ onto $E$. The set $E$ carries a natural partial order $\le$: $e\le e'$ iff $ee'=e$.
Let $E_0=\{e\in E:{\uparrow}e$ is open$\}$ stands for the set of locally minimal elements of $E$.  

For every $e\in E\setminus E_0$ let $h_e:S\to H_e^0$ be the trivial homomorphism mapping $S$ into the zero of $H_e^0$.

Next, for every $e\in E_0$ consider the homomorphism $h_e:S\to H_e^0$ defined by $$
h_e(s)=\begin{cases} es,&\mbox{if $s\in\pi^{-1}({\uparrow}e)$};\\
0,&\mbox{otherwise}
\end{cases}
$$

Taking the diagonal product of the homomorphisms $h_e$, $e\in E$, we obtain a homomorphism 
$$h=(h_e)_{e\in E}:S\to\prod_{e\in E}H^0_e,\quad h:s\mapsto (h_e(s))_{e\in E}.$$
We claim that $h$ is injective and thus an embedding of the compact semigroup $S$ into $\prod_{e\in E}H^0_e$.

Let $x,y\in S$ be two distinct points. If $\pi(x)\ne\pi(y)$, then either $\pi(x)\notin{\uparrow}\pi(y)$ or $\pi(y)\notin{\uparrow}\pi(x)$. We lose no generality assuming the first case. Consider the set $U=\{u\in E:\pi(x)\notin{\uparrow}u\}$ and note that it is open and $U={\uparrow}U$ where ${\uparrow}U=\{v\in E:\exists u\in U$ with $u\le v\}$. Also $\pi(y)\in U$.
By Proposition 1 of \cite{Hr} there is a continuous semilattice homomorphism $h:E\to\{0,1\}$ such that $\pi(y)\in h^{-1}(1)\subset{\uparrow}U$. The preimage $h^{-1}(1)$, being a compact subsemilattice of $E$, has the smallest element $e$, that belongs to $E_0$ because  $h^{-1}(1)={\uparrow}e$.

Now the definition of the homomorphism $h_e$ and the non-inclusion $\pi(x)\notin{\uparrow}e$ imply that $h_e(x)=0$ while $h_e(y)\in H_e$. Hence $h_e(x)\ne h_e(y)$ and $h(x)\ne h(y)$.

Finally consider the case $\pi(x)=\pi(y)$. Observe that the set $U=\{e\in E:xe\ne ye\}$ contains the idempotent $\pi(x)=\pi(y)$ and coincides with ${\uparrow}U$. Again applying Proposition 1 of \cite{Hr} we can find a continuous semilattice homomorphism $h:E\to\{0,1\}$ such that $\pi(x)=\pi(y)\in h^{-1}(1)\subset {\uparrow}U$. The preimage $h^{-1}(1)$, being a compact subsemilattice of $E$, has the smallest element $e$. Since $h^{-1}(1)={\uparrow}e$ is open in $E$, $e\in E_0$. It follows from $e\in U$ that $h_e(x)=ex\ne ey=h_e(y)$ and hence $h(x)\ne h(y)$.
\end{proof}

Theorem~\ref{t5} will be applied to characterize Clifford compact topological semigroups embeddable into the hypersemigroups of topological groups $G$ belonging to certain varieties of compact topological groups. A class $\mathcal G$ of topological groups is called a {\em variety} if it is closed under taking arbitrary Tychonov products, taking closed subgroups, and quotient groups by closed normal subgroups.

\begin{theorem}\label{t6} Let $\mathcal G$ be a non-trivial variety of compact topological groups. A  Clifford compact topological semigroup $S$ embeds into the hypersemigroup $\exp(G)$ of a topological group $G\in\mathcal G$ if and only if $S$ is a topological inverse semigroup whose idempotent semilattice $E$ is zero-dimensional and all maximal groups $H_e$, $e\in E$, belong to the class $\mathcal G$.
\end{theorem}

This theorem will be proved in Section~\ref{pf6} after establishing the nature of group elements in the hypersemigroups.

The classes $\mathcal H$ and $\mathcal H_0$ are closed under subdirect products but are very far from being closed under homomorphic images. We shall show that the class of continuous homomorphic images of compact Clifford semigroups $S\in\mathcal H_0$ coincides with the class of all compact Clifford inverse semigroups with Lawson idempotent semilattices. We recall that a topological semilattice $E$ is called {\em Lawson} if open subsemilattices form a base of the topology of $E$. By the fundamental Lawson Theorem \cite[Th. 2.13]{CHK2} a compact topological semilattice is Lawson if and only if the continuous homomorphisms to the $\min$-interval $[0,1]$ separate points of $S$. It is known \cite[Th. 2.6]{CHK2} that each zero-dimensional compact topological semilattice is Lawson.

\begin{proposition} A topological semigroup $S$ is a continuous homomorphic image of a compact Clifford semigroup $S_0\in\mathcal H_0$ if and only if $S$ is a compact Clifford topological inverse semigroup with Lawson idempotent semilattice.
\end{proposition}

\begin{proof} To prove the ``only if'' part, assume that a topological semigroup $S$ is the image of a compact Clifford semigroup $S_0\in\mathcal H_0$ under a continuous homomorphism $h:S_0\to S$. By Theorem~\ref{t5}(3), $S_0$ is a topological inverse Clifford semigroup with zero-dimensional idempotent semilattice $E_0$. Then $S$ is an inverse Clifford semigroup, being the homomorphic image of $S_0$, see \cite[L.II.1.10]{Pet}.
Moreover, being compact topological semigroup, $S$ is a topological inverse semigroup, see \cite{KW}, \cite{Kr} or \cite{BG}. It follows that the semigroup $E$ of idempotents of $S$ is the homomorphic image of the semilattice $E_0$. Being zero-dimensional and compact, the semilattice $E_0$ is Lawson \cite[Th.2.6]{CHK2}. Then $E$ is Lawson as the compact homomorphic image of a Lawson semilattice \cite[Th.2.4]{CHK2}.
\smallskip

To prove the ``if'' part, assume that $S$ is a compact topological inverse Clifford semigroup $S$ with Lawson semilattice $E$ of idempotents. By Corollary 1 of \cite{Hr}, $S$ embeds into a product $\prod_{\alpha A}\widehat H_\alpha$ of the cones over compact topological groups $H_\alpha$. By definition, for a compact topological group $G$ the semigroup $$\widehat H=H\times [0,1]/H\times\{0\}$$ that is the quotient semigroup of the product $H\times [0,1]$ of $H$ with the min-interval $[0,1]$ by the ideal $H\times\{0\}$ of $H\times [0,1]$.

Observe that the unit interval $[0,1]$ is the image of the standard Cantor set $C\subset[0,1]$ under a continuous monotone map $h:C\to[0,1]$ well-known under the name ``Cantor ladder''. The map $h$ can be thought as a continuous semilattice homomorphism $h:C\to[0,1]$, where both $C$ and $[0,1]$ are endowed with the operation of minimum. Then $\widehat H$ is the image of the semigroup $H\times C$, which is a compact topological inverse Clifford semigroup with zero-dimensional idempotent semilattice $C$.

Thus for each index $\alpha\in A$ we can construct a continuous surjective homomorphism $h_\alpha:S_\alpha\to \widehat H_\alpha$ of a compact topological inverse Clifford semigroup $S_\alpha$ with zero-dimensional idempotent semilattice onto the semigroups $\widehat H_\alpha$. Taking the product of those homomorphisms we obtains a continuous surjective homomorphism 
$$h:\prod_{\alpha\in A}S_\alpha\to\prod_{\alpha\in A}\widehat H_\alpha.$$
It is clear that $\prod_{\alpha\in A}S_\alpha$ is a compact topological inverse Clifford semigroup with zero-dimensional idempotent semilattice. By Theorem~\ref{t5}(2), this semigroup belongs to the class $\mathcal H_0$ and so does its subsemigroup $S_0=h^{-1}(S)$. It remains to observe that $S$ is the continuous homomorphic image of the semigroup $S_0\in\mathcal H_0$.
\end{proof}

 This proposition yields many examples of compact Clifford semigroups $S\notin\mathcal H$ that are continuous homomorphic images of compact Clifford semigroups $S_0\in\mathcal H_0\subset \mathcal H$. We have also a non-Clifford example.

\begin{example} The holomorph $\Hol(E_3)=E_3\leftthreetimes Aut(E_3)$ of the 3-element semilattice $E_3=\{e,f,ef\}$ belongs to the class $\mathcal H_0$ but contains the 2-element ideal $I=\{ef\}\times Aut(E_3)$ such that quotient semigroup $\Hol(E_3)/I$ is isomorphic to the 5-element Brandt semigroup $B(\IZ_1,2)$ and thus does not belong to the class $\mathcal H$.
\end{example}

The remaining part of the paper is devoted to the proofs of the results announced in the Introduction.

\section{Semidirect products of topological semigroups}\label{s3}

In this section we shall prove that the class $\mathcal H$ is closed under semidirect products with Abelian topological groups. 

Let $G$ be a topological group. By a topological {\em $G$-semigroup} we understand a topological semigroup $S$ endowed with a homomorphism $\sigma:G\to Aut(S)$ of $G$ to the group of topological automorphisms of $S$ such that the induced action $\tilde \sigma:G\times S\to S$, $\tilde\sigma:(g,s)\mapsto \sigma(g)(s)$, is continuous. It will be convenient to denote the element $\sigma(g)(s)$ by the symbol $gs$.

The {\em semidirect product} $S\leftthreetimes^\sigma G$ of  a topological $G$-semigroup $S$ with $G$ is the topological semigroup whose underlying topological space is $S\times G$ and the semigroup operation is defined by $(s,g)*(s',g')=(s\cdot gs',gg')$. If the action $\sigma$ of the group $G$ on $S$ is clear from the context, then we shall omit the symbol $\sigma$ and will write $S\leftthreetimes G$ instead of $S\leftthreetimes^\sigma G$.

The following  proposition describes some algebraic properties of semidirect products.

\begin{proposition}\label{semidir} Let $S\leftthreetimes G$ be the semidirect product of a topological $G$-semigroup $S$ and a topological group $G$.
\begin{enumerate}
\item $S\leftthreetimes G$ is a (topological) inverse  semigroup if and only if $S$ is a (topological) inverse  semigroup;
\item $S\leftthreetimes G$ is a topological group if and only if $S$ is a topological group;
\item $S\leftthreetimes G$ is an inverse Clifford semigroup if and only if $S$ is an inverse Clifford semigroup and $ge=e$ for any $g\in G$ and any idempotent $e$ of $S$.
\end{enumerate} 
\end{proposition}

\begin{proof} First observe that $S$ can be identified with the subsemigroup $S\times\{e\}$ of $S\leftthreetimes G$ where $e$ is the unique idempotent of $G$.
\smallskip

1. Assume that $S$ is an inverse semigroup. To show that $S\leftthreetimes G$ is an inverse semigroup we should check that the idempotents of $S\leftthreetimes G$ commute and each element $(s,g)\in S\leftthreetimes G$ has an inverse. For this observe that $(g^{-1}s^{-1},g^{-1})$ is an inverse element to $(s,g)$. Indeed, 
$$(s,g)*(g^{-1}s^{-1},g^{-1})*(s,g)=(ss^{-1},e)(s,g)=(ss^{-1}s,g)=(s,g).$$
By analogy we can check that $$(g^{-1}s^{-1},g^{-1})(s,g)(g^{-1}s^{-1},g^{-1})=(g^{-1}s^{-1},g^{-1}).$$

Observe that an element $(s,g)$ is an idempotent of the semigroup $S\leftthreetimes G$ is and only if $s$ and $g$ are idempotents. This observation easily implies that the idempotents of the semigroup $S\leftthreetimes G$ commute (because the  idempotents of $S$ commute).

If $S$ is a topological inverse semigroup, then the map $(\cdot)^{-1}:S\to S$, $(\cdot)^{-1}:s\mapsto s^{-1}$ is continuous. The continuity of this map can be used to show that the map 
$$(\cdot)^{-1}: S\leftthreetimes G\to S\leftthreetimes G,\quad (\cdot)^{-1}:(s,g)\mapsto (g^{-1}s^{-1},g^{-1})$$is continuous too.

Next, assume that $S\leftthreetimes G$ is an inverse semigroup. Given any element $s$ consider the element $x=(s,e)\in S\leftthreetimes G$ and find its inverse $x^{-1}=(s',g)$ in 
$S\leftthreetimes G$. It follows from $(s,e)(s',g)(s,e)=xx^{-1}x=x=(s,e)$ that $g=e$ and then $ss's=s$ and $s'ss'=s$, which means that $s'$ is the inverse element to $s$ in the semigroup $S$. Since the idempotents of $S\leftthreetimes G$ commute and lie in the subsemigroup $S\times\{e\}$, the idempotents of $S$ commute too, which yields that $S$ is an inverse semigroup. 

If  $S\leftthreetimes G$ is a topological inverse semigroups, then $S$ is a topological inverse semigroup, being a subsemigroup of $S\leftthreetimes G$.
\smallskip

2. The second item follows from the first one and the fact and a topological semigroup is a topological group if and only if it is a topological inverse semigroup with a unique idempotent.
\smallskip

3. Assume that the semigroup $S$ is inverse and Clifford, and $G$ acts trivially on the idempotents of $S$. By the first item, $S\leftthreetimes G$ is an inverse semigroup. So it remains to prove that $xx^{-1}=x^{-1}x$ for all 
$x=(s,g)\in S\leftthreetimes G$. Observe that $x^{-1}=(g^{-1}s^{-1},g^{-1})$ and thus 
$$\begin{aligned}
x^{-1}x=&\;(g^{-1}s^{-1},g^{-1})(s,g)=(g^{-1}s^{-1}g^{-1}s,e)=(g^{-1}(s^{-1}s),e)=\\=&\;(s^{-1}s,e)=(ss^{-1},e)=(s,g)(g^{-1}s^{-1},g^{-1})=xx^{-1}.
\end{aligned}$$
Here we used that $G$ acts trivially on the idempotents of $S$ and hence $g^{-1}(s^{-1}s)=s^{-1}s$. We also used that fact that $g^{-1}:s\mapsto g^{-1}s$ is an automorphism of the semigroup $S$ and thus $g^{-1}(s^{-1}s)=(g^{-1}s^{-1})(g^{-1}s)$. 
Now assume that the semigroup $S\leftthreetimes G$ is Clifford and inverse. Then $S$ is Clifford, being a subsemigroup of $S\leftthreetimes G$. It remains to show that $G$ acts trivially on the idempotents of $S$. Take any idempotent $s\in S$, any $g\in G$, and consider the element  $x=(s,g)$ and its inverse $x^{-1}=(g^{-1}s^{-1},g^{-1})$. Since $S\leftthreetimes G$ is Clifford, $xx^{-1}=x^{-1}x$, which implies that 
$$
\begin{aligned}
x^{-1}x=&\;(g^{-1}s^{-1},g^{-1})(s,g)=(g^{-1}s^{-1}g^{-1}s,e)=
(g^{-1}s^{-1}s,e)=\\=&\;(g^{-1}s,e)=xx^{-1}=(ss^{-1},e)=(s,e)\end{aligned}$$ and thus $gs=s$.
\end{proof}

If $S$ is a topological $G$-semigroup, then $\exp(S)$ has a structure of a topological $G$-semigroup with respect to the induced action $$G\times\exp(S)\to\exp(S),\quad (g,K)\mapsto gK=\{gs:s\in K\}.$$ Thus is it legal to consider the semidirect product $\exp(S)\leftthreetimes G$.

The proof of the following proposition is easy and is left to the reader.

\begin{lemma}\label{l1} The map 
$$E:\exp(S)\leftthreetimes G\to \exp(S\leftthreetimes G),\; E:(K,g)\mapsto K\times\{g\}$$is a topological embedding of the topological semigroups.
\end{lemma}

For a topological semigroup $S$ consider the Tychonov power $S^G$ as a topological $G$-semigroup with the following action of $G$:
$$(g,(s_\alpha)_{\alpha\in G})\mapsto (s_{g\alpha})_{\alpha\in G}.$$
A homomorphism $h:S\to S'$ between two $G$-semigroups is called {\em $G$-equivariant} if $h(gs)=gh(s)$ for every $g\in G$ and $s\in S$. The proof of the following lemma also is left to the reader.

\begin{lemma}\label{l2} For any topological semigroup $H$ the map
$$E:\exp(H)^G\to\exp(H^G),\; E:(K_\alpha)_{\alpha\in G}\mapsto \prod_{\alpha\in G}K_\alpha$$
is a $G$-equivariant embedding of the corresponding $G$-semigroups.
\end{lemma}

The following immediate lemma helps to transform semigroup embedding into $G$-equivariant embedding.

\begin{lemma}\label{l3} Let $G$ be an Abelian topological group. If $f:S\to H$ is an embedding of a topological $G$-semigroup $H$ into a topological semigroup $H$, then the map
$$F:S\to H^G,\; F:s\mapsto (f(gs))_{g\in G}$$
is a $G$-equivariant embedding of the $G$-semigroup $S$ into the $G$-semigroup $H^G$.
\end{lemma}

Finally we are able to prove the second item of Theorem~\ref{t1}.

\begin{theorem} Let $G$ be an Abelian topological group. If a topological $G$-semigroup $S$ embeds into the hypersemigroup $\exp(H)$ of a topological group $H$, then the semidirect product $S\leftthreetimes G$ embeds into the hypersemigroup $\exp(H^G\leftthreetimes G)$ of the topological group $H^G\leftthreetimes G$.
\end{theorem}

\begin{proof} Let $f:S\to\exp(H)$ be an embedding. By Lemmas~\ref{l3} and \ref{l2}, the map
$$F:S\to \exp(H^G),\; F:s\mapsto \prod_{\alpha\in G}f(\alpha s)$$
is a $G$-equivariant embedding. The $G$-equivariantness of $F$ guarantees that the map
$$E:S\leftthreetimes G\to \exp(H^G)\leftthreetimes G,\quad E:(s,g)\mapsto (F(s),g)$$ is an embedding of the corresponding topological semigroups. Finally, applying Lemma~\ref{l1} we see that the semigroup $S\leftthreetimes G$ admits an embedding into the hypersemigroup $\exp(H^G\leftthreetimes G)$ of the topological group $H^G\leftthreetimes G$.
\end{proof}

\section{Idempotents and invertible elements of the hypersemigroups}\label{s5}

In this section given a topological group $G$ we characterize idempotent and related special elements of the hypersemigroup $\exp(G)$. We recall that an element $x$ of a semigroup $S$ is called 
\begin{itemize}
\item an {\em idempotent} if $xx=x$;
\item {\em regular} if there is an element $y\in S$ such that $xyx=x$;
\item ({\em uniquely}) {\em invertible} if there is a (unique) element $x^{-1}\in S$ such that $xx^{-1}x=x$ and $x^{-1}xx^{-1}=x^{-1}$;
\item {\em a group element} if $x$ lies in some subgroup of $S$.
\end{itemize}

It is possible to prove our results in a more general setting of cancellative topological semigroups. We recall that a semigroup $S$ is {\em cancellative} if for any $x,y,z\in S$ the equality $xz=yz$ implies $x=y$ and the equality $zx=zy$ implies $x=y$. It is easy to check that the invertible elements of a cancellative semigroup form a subgroup.

\begin{proposition}\label{p3} Let $X$ be a cancellative topological semigroup. A non-empty compact subset $K\subset X$ is 
\begin{enumerate}
\item an idempotent of the semigroup $\exp(X)$ if and only if $K$ is a compact subgroup of $X$;
\item a regular element of the semigroup $\exp(X)$ if and only if $K$ uniquely invertible in $\exp(X)$ if and only if $K=Hx$ for some compact subgroup $H\subset X$ and some invertible element $x\in X$;
\item a group element in $\exp(X)$ if  and only if $K=Hx=xH$ for some compact subgroup $H\subset X$ and some invertible element $x\in X$.
\end{enumerate}
\end{proposition}

\begin{proof} 1. If a compact subset $K\subset X$ is an idempotent of the semigroup $\exp(X)$ that is $KK=K$, then $K$ is a compact cancellative semigroup.
It is known \cite[Th. 1.10]{CHK1} that a compact cancellative semigroup is a group. 
If $K$ is subgroup of $X$ then $KK=K$. 
\smallskip

2. Assume that $K\in\exp(X)$ is a regular element of the semigroup $\exp(X)$ which means that $KAK=K$ for some non-empty compact subset $A\subset X$. Fix any elements $x\in K$ and $a\in A$.  The set $KA$, being an idempotent of the semigroup $\exp(X)$, coincides with some compact subgroup  $H$ of $X$. We claim that $K=Hx$ and the element $x$ is invertible in $X$. Observe that 
$Hx\subset HK=KAK=K$ and thus $Hxa\subset KA=H$, which implies that $xa\in H$ is invertible.  Consequently, $xa(xa)^{-1}=e=(xa)^{-1}xa$ which means that $x$ and $a$ are invertible. It follows from $Ka\subset KA=H$ that 
$$K\subset Ha^{-1}=Ha^{-1}x^{-1}x=H(xa)^{-1}x\subset HHx=Hx\subset K$$ and thus $K=Hx$.

To show that $K$ is uniquely invertible, assume additionally that $AKA=A$. In this case $A=AKA\supset aKa=aHxa=aH=x^{-1}xaH=x^{-1}H$. On the other hand, the equality $KAK=K$ implies $xAx\subset Hx$ and $A\subset x^{-1}H$. Therefore $A=x^{-1}H$ is a unique inverse element to $K$.
\smallskip

3. If $K=Hx=xH$ for some compact subgroup $H\subset X$ and some invertible element $x\in X$, then for the element $K^{-1}=x^{-1}H=Hx^{-1}$ we get $K^{-1}K=KK^{-1}=H$, which implies that $K$ is a group element of $\exp(X)$. Conversely, if $K$ is a group element, then $KK^{-1}=K^{-1}K=H$ for some compact subgroup $H\subset X$ and $K=Hx$ for some invertible element $x\in X$ (because $K$ is regular). Since $H=K^{-1}K=x^{-1}HHx=x^{-1}Hx$, we get $xH=Hx$.
\end{proof}

Theorem~\ref{t2} is a particular case of the following more general

\begin{theorem} Let $X$ be a cancellative topological semigroup and $G$ be the subgroup of invertible elements of $X$. Let $S$ be an algebraically  regular subsemigroup of $\exp(X)$ and $E$ be the set of idempotents of $S$.
\begin{enumerate}
\item The semigroup $S$ is inverse and $S\subset \exp(G)$.
\item If $G$ is a topological group, then $S$ is a topological inverse semigroup.
\item An element $x\in S$ is an idempotent if $x^2x^{-1}$ is an idempotent.
\item Any distinct conjugate idempotents of $S$ are incomparable. 
\item The set $E$ is a closed commutative subsemigroup of $S$ and for every $e\in E$ the upper cone ${\uparrow}e=\{f\in E:ef=e\}$ is totally disconnected.
\end{enumerate}
\end{theorem}

\begin{proof} 1. Let $S$ be a regular subsemigroup of $\exp(X)$. It follows from Proposition~\ref{p3} that each element $K\in S$, being regular, is equal to $Hx$ for some compact subgroup $H\subset G$ and some invertible element $x\in X$. Then $K=Hx\subset G$ and hence $K\in\exp(G)\subset\exp(X)$. By Proposition~\ref{p3}, $K$ is uniquely invertible in $\exp(X)$ and hence in $S$, which means that $S$ is an inverse semigroup. Moreover, the inverse $K^{-1}$ to $K$ in $S$ can be found by the natural formula: $K^{-1}=\{x^{-1}:x\in K\}$.
\smallskip

2. If the subgroup $G$ of invertible elements of $X$ is a topological group, then the inversion 
$$(\cdot)^{-1}:\exp(G)\to\exp(G),\;(\cdot)^{-1}:K\mapsto K^{-1}$$
is continuous with respect to the Vietoris topology on $\exp(G)$ and consequently, the inversion map of $S$ is continuous as well, which yields that $S$ is a topological inverse semigroup.
\smallskip

3. Let $K\in S$ be an element such that $K^2K^{-1}$ is an idempotent in $S$ and hence is a compact subgroup of $X$. By Proposition~\ref{p3}, $K=Hx$ for some compact subgroup $H$ of $X$ and some invertible element $x\in X$. Then $K^2K^{-1}=HxHxx^{-1}H=HxH$. The set $K^2K^{-1}$, being a subgroup of $X$, contains the neutral element $1$ of $X$. Then $1\in K^2K^{-1}=HxH$ and hence $x\in H$, which implies that $K=Hx=H$ is an idempotent in $\exp(X)$ and $S$.
\smallskip

4. Let $E,F$ be two distinct conjugate idempotents of the semigroup $S$. Find an element $K\in S$ such that $E=KFK^{-1}$ and $F=K^{-1}EK$. By Proposition~\ref{p3}, find a compact subgroup $H$ of $X$ and an invertible element $x\in X$ such that $K=Hx$.
We claim that $E=xFx^{-1}$. Indeed, the inclusion 
$$x^{-1}Hx=x^{-1}HHx=K^{-1}K\subset K^{-1}EK=F$$ implies $$E=KFK^{-1}=HxFx^{-1}H=xx^{-1}HxFx^{-1}Hxx^{-1}\subset xFFFx^{-1}=xFx^{-1}.$$
On the other hand,$$H=Hxx^{-1}H\subset HxFx^{-1}H=KFK^{-1}=E$$implies
$$F=K^{-1}EK=x^{-1}HEHx\subset x^{-1}EEEx=x^{-1}Ex$$ and hence $xFx^{-1}\subset E$.
\smallskip

5. Since $S$ is an inverse semigroup, the set $E$ of idempotents of $S$ is a commutative subsemigroup of $S$, see \cite[II.1.2]{Pet}. To show that $E$ is closed in $S$, pick any element $K\in S\setminus E$. By Proposition~\ref{p3}, $K=Hx$ for a compact subgroup $H\subset X$ and an invertible element $x\in X$. Since $K$ is not an idempotent, $Hx$ is not a subgroup, which means that the neutral element $1$ of $H$ does not belong to $Hx$. Let $U=X\setminus\{x\}$ and observe that $U^+=\{C\in\exp(X):1\notin C\}$ is a neighborhood of $K$ in $\exp(X)$ that contains no subgroup of $X$ and hence does not intersect the set $E$. 

 Now given an idempotent $H\in \mathcal E$ we shall prove that the upper cone ${\uparrow}H=\{E\in \mathcal E:HE=H\}$ of $H$ is totally disconnected. By Proposition~\ref{p3}, $H$ is a compact subgroup of $X$.
It follows that ${\uparrow}H\subset\exp(H)$. The total disconnectedness of ${\uparrow}H$ will be proven as soon as given two distinct elements $E_0,E_1\in{\uparrow}H$ we find a closed-and-open subset $\U\subset{\uparrow}H$ such that $E_0\in\U$ but $E_1\notin\U$. We loose no generality assuming that $X=H$ and hence $\mathcal E={\uparrow}H\subset\exp(H)$.

We first consider the special case when $H$ is a Lie group. Without loss of generality $E_0\not\subset E_1$ and hence $E_0\notin {\downarrow}E_1=\{E\in\mathcal E: E\subset E_1\}$. So, it remains to prove that the lower cone ${\downarrow}E_1$ is closed-and open in $\mathcal E$. The closedness of ${\downarrow}E_1$ follows from the continuity of the semigroup operation and the equality ${\downarrow}E_1=\{E\in\mathcal E:EE_1=E_1\}$. To prove that ${\downarrow}E_1$ is open in $\mathcal E$, take any $K\in{\downarrow}E_1$. The set $K\in\exp(H)$, being an idempotent of the semigroup $\mathcal E$ is a closed subgroup of $H$. 

By Corollary~II.5.6 of \cite{Bre} the subgroup $K$ of the compact Lie group $H$ has an open neighborhood $O(K)\subset H$ such that for each compact subgroup $C\subset O(K)$ satisfies the inclusion $xCx^{-1}\subset K$ for a suitable point $x\in H$. We shall derive from this that $C=K$ provided $C\supset K$. Indeed, $C\supset K$ and $xCx^{-1}\subset K$ imply $xKx^{-1}\subset xCx^{-1}\subset K$. Being a homeomorphic copy of the group $K$, the subgroup $xKx^{-1}\subset K$ must coincide with $K$ (it has the same dimension and the same number of connected components). Consequently, $xCx^{-1}=K$ and hence $C$, being homeomorphic to its subgroup $K$, coincides with $K$ too.

The continuity of the semigroup operation of $\mathcal E$ yields a neighborhood $O_1(K)\subset\mathcal E$ of $K$ such that $EK\subset O(K)$ for each  $E\in O_1(K)$.
We claim that $O_1(K)\subset {\downarrow}E_1$. Take any element $E\in O_1(K)$ and observe that the product $EK$, being an idempotent in the semigroup $\mathcal E$,  is a compact subgroup of $H$ containing the subgroup $K$. Now the choice of the neighborhood $O(K)$ guarantees that $E\subset EK\subset K\subset E_1$ and hence $E\in{\downarrow}E_1$. This proves that $O_1(K)\subset{\downarrow}E_1$, witnessing that ${\downarrow}E_1$ is open in $\mathcal E$.

Now we are able to finish the proof assuming that $H$ is an arbitrary compact topological group. Given distinct elements $E_0,E_1\in\mathcal E\subset\exp(H)$ we should find an closed-and-open subset $\U\subset\mathcal E$ containing $E_0$ but not $E_1$. 
The topological group $H$, being compact, is the  limit of an inverse spectrum consisting of compact Lie groups. Consequently, we can find a continuous homomorphism $h:H\to L$ onto a compact Lie group $L$ such that $h(E_0)$ and $h(E_1)$ are distinct subgroups of $L$. It follows that $h(\mathcal E)=\{h(E):E\in\mathcal E\}$ is an idempotent semigroup of the hypersemigroup $\exp(L)$. Now the particular case considered above yields a closed-and-open subset $\V\subset h(\mathcal E)$ containing $h(E_0)$ but not $h(E_1)$. By the continuity of the homomorphism $h$ the set $\U=\{K\in\mathcal E:h(K)\in\V\}$ is closed-and-open in $\mathcal E$. It contains $E_0$ but not $E_1$. This proved the total disconnectedness of the upper cone ${\uparrow}H$.
\end{proof}

\section{Proof of Theorem~\ref{t6}}\label{pf6}

In this section we will prove the Theorem~\ref{t6}. Given a Clifford compact topological semigroup $S$ and a non-trivial variety $\mathcal G$ of compact topological groups we should prove that $S$ embeds into the hypersemigroup $\exp(G)$ of a topological group $G\in\mathcal G$ if and only if $S$ is a topological inverse semigroup whose idempotent semilattice $E$ is zero-dimensional and all maximal groups $H_e$, $e\in E$, belong to the class $\mathcal G$.

To prove the ``if'' part, assume that $S$ is a compact Clifford topological inverse semigroup whose idempotent semilattice $E$ is zero-dimensional and all maximal groups $H_e$, $e\in E$, belong to the class $\mathcal G$. For every $e\in E$ let $\widetilde H_e=H_e$ if $H_e$ is not trivial and $\widetilde H_e\in\mathcal G$ be any non-trivial compact group if $H_e$ is trivial (such a group $\tilde H_e$ exists because the variety $\mathcal G$ is not trivial). Since $\mathcal G$ is closed under taking Tychonov products, the compact topological group $G=\prod_{e\in E}\widehat H_e$ belongs to $\mathcal G$. Finally, by Theorem~\ref{t5}(5), the semigroup $S$ embeds into $\exp(G)$.
\smallskip

To prove the ``only if'' part, assume that $S$ embeds into the hypersemigroup $\exp(G)$ over a topological group $G\in\mathcal G$.
By Theorem~\ref{t5}(3), $S$ is a compact topological inverse Clifford semigroup with zero-dimensional idempotent semilattice $E$. It remains to show that each maximal group $H_e$, $e\in E$, of $S$ belongs to $\mathcal G$. The embedding of $S$ into $\exp(G)$ induces an embedding $h:H_e\to\exp(G)$. The image $H_0=h_(e)$, being an idempotent in $\exp(G)$, is a compact subsemigroup of $G$ and thus a compact subgroup of $G$ according to Theorem 1.10 \cite{CHK1}. The same is true for the semigroup $H=\bigcup_{x\in H_e}h(x)$. It is a compact subgroup of $G$. We claim that $H_0$ is a normal subgroup of $H$.

Indeed, for any $x\in H$ we can find a point $z\in H_e$ with $x\in h_e(z)$. It  follows from (the proof of) Proposition~\ref{p3}(3) that $h_e(z)=xH_0=H_0x$ $xH_0x^{-1}=H_0$, witnessing that the subgroup $H_0$ is normal in $H$.

Let $\pi:H\to H/H_0$ be the quotient homomorphism. It follows from Proposition~\ref{p3}(3) that the composition $\pi\circ h_e:H_e\to H/H_0$ is a bijective continuous homomorphism. Because of the compactness of $H_e$, the group $H_e$ is isomorphic to $H/H_0$, which, being the quotient group of the closed subgroup $H$ of the group $G\in\mathcal G$ belongs to the variety $\mathcal G$.

\section{Acknowledgements}

The authors express their sincere thanks to Oleg Gutik for fruitful discussions on the subject of the paper.

\end{document}